\documentclass[a4paper,12pt]{amsart}

\usepackage[margin=2cm]{geometry}
\usepackage[alphabetic]{amsrefs}
\usepackage{amssymb}

\usepackage{amsmath, mathtools} 
\usepackage{amsthm} 
\usepackage{eucal} 
\usepackage{pdfpages}
\usepackage{mathrsfs}
\usepackage{hyperref}

\theoremstyle{plain}

\newtheorem{theorem}{Theorem}[section]
\newtheorem{lem}[theorem]{Lemma}
\newtheorem{cor}[theorem]{Corollary}
\newtheorem{prop}[theorem]{Proposition}
\theoremstyle{definition}

\newcommand{\R}{\mathbb{R}}
\newcommand{\N}{\mathbb{N}}

\newcommand{\Ns}{\mathscr{N}_{s}}
\newcommand{\Q}{\mathcal{Q}}

\renewcommand{\leq}{\leqslant}
\renewcommand{\le}{\leqslant}
\renewcommand{\geq}{\geqslant}
\renewcommand{\ge}{\geqslant}

\begin{document}

\title[Linear theory for
a mixed operator with Neumann conditions]{Linear theory for
a mixed operator \\ with Neumann conditions}\thanks{
{\em Serena Dipierro}:
Department of Mathematics
and Statistics,
University of Western Australia,
35 Stirling Hwy, Crawley WA 6009, Australia.
{\tt serena.dipierro@uwa.edu.au}\\
{\em Edoardo Proietti Lippi}: 
Department of Mathematics and Computer Science,
University of Florence,
Viale Morgagni 67/A, 50134 Firenze, Italy.
{\tt edoardo.proiettilippi@unifi.it}
\\
{\em Enrico Valdinoci}:
Department of Mathematics
and Statistics,
University of Western Australia,
35 Stirling Hwy, Crawley WA 6009, Australia. {\tt enrico.valdinoci@uwa.edu.au}\\
The authors are members of INdAM.
The first and third authors
are members of AustMS and
are supported by the Australian Research Council
Discovery Project DP170104880 NEW ``Nonlocal Equations at Work''.
The first author is supported by
the Australian Research Council DECRA DE180100957
``PDEs, free boundaries and applications''. Part of this
work was carried
out during a very pleasant and fruitful visit of the second author to the
University of Western Australia, which we thank for the warm hospitality.}

\author{Serena Dipierro}
\author{Edoardo Proietti Lippi}
\author{Enrico Valdinoci}


\begin{abstract}
We consider here
a new type of mixed local and nonlocal equation under
suitable Neumann conditions. We discuss the spectral properties
associated to a weighted eigenvalue problem and present
a global bound for subsolutions.

The Neumann condition that we take into account
comprises, as a particular case, the one that has been recently introduced
in~[S.~Dipierro, X.~Ros-Oton, E.~Valdinoci,
Rev. Mat. Iberoam. (2017)].

Also, the results that we present here find a natural application
to a logistic equation motivated by biological problems
that has been recently considered in~[S.~Dipierro, E. Proietti Lippi, E.~Valdinoci,
preprint (2020)].
\end{abstract}
\maketitle

\section{Introduction}

The goal of this article is to discuss the spectral properties
and the $L^\infty$-bounds associated to a mixed local and nonlocal problem,
also in relation to some concrete motivations arising from
population dynamics and mathematical biology.
The methodology that we exploit here relies on functional analysis
and methods from (classical and nonlocal) partial differential
equations. Given the mixed character of the operator taken into account
and the new set of external conditions,
the standard mathematical framework to deal with
partial and integro-differential equations needs to be conveniently
modified to suit this new scenario.\medskip

More specifically, in~\cite{VERO}, we have introduced
a new set of nonlocal Neumann conditions,
extending those previously set forth in~\cite{MR3651008},
with the aim of dealing with a mathematical problem
motivated by ethology and biology.
More specifically, in~\cite{VERO}
a biological population was taken into consideration
within an environment which could be partially hostile.
The population competes for the resources
via a logistic equation and diffuses by a possible combination
of classical and nonlocal dispersal processes
(a detailed derivation of the diffusion model
is also presented in the appendix of~\cite{VERO}).

The population can be also provided by an additional birth growth
due to pollination, and the main question targeted in~\cite{VERO}
is whether or not it is possible to {\em rearrange the given
environmental resources (within given upper and lower constraints)
to allow for the survival of the species}.\medskip

The nonlinear
mathematical analysis developed in~\cite{VERO}
also relies on some auxiliary results
from the linear theory, such as {\em spectral decompositions}
and {\em uniform bounds for subsolutions}, which
have their independent interest. We collect here these
results, providing full proofs in detail.
\medskip

The setting in which we work is the following.
We let~$s\in(0,1)$ and~$\alpha$, $\beta\in[0,+\infty)$
with~$\alpha+\beta>0$,
and we consider the mixed operator
\begin{equation}\label{9uyhifd774} -\alpha\Delta +\beta(-\Delta)^s.\end{equation}
As customary, the operator~$(-\Delta)^s$
is the fractional Laplacian
$$ (-\Delta)^s u(x):=\frac12\,\int_{\R^n}\frac{2u(x)
-u(x+\zeta)-u(x-\zeta)}{|\zeta|^{n+2s}}\,d\zeta,$$
where other normalization constants have
been removed to ease the notation
(in any case, additional normalizing
constants do not affect our arguments,
and they can also be comprised into the parameter~$\beta$
in~\eqref{9uyhifd774} if one wishes to do so).\medskip

As a matter of fact, the theory that we develop
here, as well as in~\cite{VERO},
works in greater generality (e.g., one
can replace the fractional Laplacian with
a more general integro-differential operator
with only minor modifications in the main proofs),
but we rather limit ourselves to the paradigmatic
case of the fractional Laplacian for the sake of simplicity
in the exposition. Moreover, the results
obtained are new even in the case of ``purely nonlocal
diffusion'', i.e. when~$\alpha=0$
in~\eqref{9uyhifd774}.
\medskip

In terms of theory and applications,
we recall that operators with mixed classical
and fractional orders have been studied
under different points of views,
see for instance~\cite{MR2095633, MR2180302, MR2129093, MR2243708,
MR2422079, MR2542727, 
MR2653895, MR2911421, MR2912450,
MR2928344, MR2963799,
MR3051400, MR3194684, MR3485125,
MR3950697, MR3912710,
2017arXiv170605306D, 2018arXiv181107667D, 2019arXiv190702495A,
biagvecc, ABATANGELO, CABRE}
and the references therein.
Besides their clear mathematical interest,
these operators find natural applications
in biology, in view of the long-jump dispersal
strategies followed by several species,
as confirmed by a number of experimental data, see e.g.~\cite{NATU},
and theoretically studied under several perspectives,
see e.g.~\cite{MR1636644, MR2332679, MR2411225, MR2601079, MR2897881, 
MR2924452, MR3026598, MR3035974, MR3082317, MR3169773, MR3285831,
MR3498523, 3579567, MR3590646,
MR3590678, MR3639140, MR3771424}
(other concrete applications arise in plasma physics,
see~\cite{PhysRevE.87.063106} and the references therein).
\medskip

As usual, the mathematical framework in~\eqref{9uyhifd774}
is endowed by a spatial domain
on which the corresponding equation takes place.
For this,
we take a bounded open set~$\Omega\subset\R^n$
of class~$C^1$.
When~$\beta=0$, we take the additional hypothesis that
\begin{equation}\label{connected}
{\mbox{$\Omega$ is connected.}}\end{equation}
{F}rom the biological point of view, $\Omega$
represents the natural environment inhabited by
a given biological
population, whose density is described
by a function~$u:\R^n\to\R$
(as customary in nonlocal problems,
one has to prescribe functions in all of the space
to make sense of the fractional diffusive operators). \medskip

We prescribe external conditions to~$u$ in order
to make~$\Omega$ an ecological niche.
To this end, see~\cite{VERO}, we
set a variational formulation related to the
operator in~\eqref{9uyhifd774} which
endows the equation in the set~$\Omega$
with a suitable Neumann condition. The functional
space that we consider is
\begin{equation}\label{Xdefab}
X_{\alpha,\beta}=X_{\alpha,\beta}(\Omega):=
\begin{dcases}
H^1(\Omega)     & {\mbox{ if }} \;\beta=0,
\\
H^s_\Omega		& {\mbox{ if }} \;\alpha=0,
\\
H^1(\Omega)\cap H^s_\Omega & {\mbox{ if }}\; \alpha \beta\neq 0,
\end{dcases}
\end{equation}
where $$
H^s_\Omega:=
\left\lbrace  
u:\R^n\to\R \;{\mbox{ s.t. }}\;
u\in L^2(\Omega) \;{\mbox{ and }}\;
\iint_\Q\frac{|u(x)-u(y)|^2}{|x-y|^{n+2s}}\,dx\,dy <+\infty
\right\rbrace,
$$
and~${\mathcal{Q}}$ is the cross-shaped set
on~$\Omega$ given by
$$ {\mathcal{Q}}:=\big(\Omega\times\Omega\big)
\cup\big(\Omega\times(\R^n\setminus\Omega)\big)\cup
\big((\R^n\setminus\Omega)\times\Omega\big).$$
We observe
that~$X_{\alpha,\beta}$ is a Hilbert 
space with respect to the scalar product
\begin{equation}\label{scalar}\begin{split}
(u,v)_{X_{\alpha,\beta}}&\;:=\int_\Omega u(x)v(x)\,dx+
\alpha \int_\Omega \nabla u(x) \cdot\nabla v(x)\,dx
 \\
&\qquad+\frac{\beta}{2}
\iint_\Q\frac{(u(x)-u(y))(v(x)-v(y))}{|x-y|^{n+2s}}\,dx\,dy,
\end{split}\end{equation}
for every $u,v\in X_{\alpha,\beta}$.

We also define the seminorm
\begin{equation}\label{seminorm}
[u]^2_{X_{\alpha,\beta}}:=\frac\alpha2 
\int_\Omega |\nabla u(x)|^2\,dx
+\frac{\beta}{4}
\iint_\Q\frac{|u(x)-u(y)|^2}{|x-y|^{n+2s}}\,dx\,dy.
\end{equation}
Given~$f\in L^2(\Omega)$, we say
that~$u\in X_{\alpha,\beta}$ is a solution of
\begin{equation}\label{LAWEAKXA-S} -\alpha\Delta u+\beta(-\Delta)^s u=f\qquad{\mbox{ in }}\;\Omega\end{equation}
with~$(\alpha,\beta)$-Neumann condition
if
\begin{equation}\label{LAWEAKXA}
\alpha \int_\Omega \nabla u(x) \cdot\nabla v(x)\,dx
+\frac{\beta}{2}
\iint_\Q\frac{(u(x)-u(y))(v(x)-v(y))}{|x-y|^{n+2s}}\,dx\,dy
=\int_\Omega f(x)\,v(x)\,dx,\end{equation}
for every~$v\in X_{\alpha,\beta}$.

We remark that, formally, the external condition in~\eqref{LAWEAKXA}
can be detected by taking~$v$ with~$v=0$ in~$\R^n\setminus\overline{\Omega}$
(which produces a normal derivative prescription along~$\partial\Omega$)
and then by taking~$v=0$ in~$\overline\Omega$
(which produces a nonlocal prescription in~$\R^n\setminus\overline\Omega$):
that is, formally, the external condition in~\eqref{LAWEAKXA}
can be written in the form
\begin{equation}\label{NEU-3}
\begin{dcases}
\Ns u=0   & \qquad{\mbox{in }}\; \R^n \setminus \overline\Omega,
\\
\displaystyle\frac{\partial u}{\partial \nu}=0   &\qquad{\mbox{on }} \;\partial\Omega,\end{dcases}
\end{equation}
where~$\nu$ is the exterior normal to~$\Omega$,
and we use the notation
\begin{equation}\label{DRSPV}
\Ns u(x):=\int_\Omega \frac{u(x)-u(y)}{|x-y|^{n+2s}}\,dy\qquad\qquad
{\mbox{for every }}\;x\in\R^n \setminus \overline\Omega,\end{equation}
and the first condition in~\eqref{NEU-3}
being dropped when~$\alpha=0$, the second condition
in~\eqref{NEU-3}
being dropped when~$\beta=0$.\medskip

We recall that the nonlocal Neumann prescription
in~\eqref{DRSPV} is precisely the one introduced in~\cite{MR3651008}
in light of probabilistic consideration
(i.e., a particle following a~$\frac{s}2$-stable
process is sent back to the original domain
by following the same process).
Also, as shown in~\cite{MR3651008},
the setting in~\eqref{DRSPV} provides a coherent
functional analysis setting.\medskip

In the situation treated in this paper,
this setting is superimposed to a classical framework
when~$\alpha\ne0$: in particular,
we remark that, when~$\alpha\ne0$ and~$\beta\ne0$,
both the prescriptions in~\eqref{NEU-3}
are in force, but they do not cause any overdetermined
conditions, and indeed, as shown in~\cite{VERO},
the notion of solutions in this case is well-posed.\medskip

Moreover, we stress that the setting in~\eqref{LAWEAKXA}
provides a ``zero-flux''
condition, in the sense that if~\eqref{LAWEAKXA-S}
has a solution, then necessarily
\begin{equation}\label{ZERAJS-jPP} \int_\Omega f(x)\,dx=0,\end{equation}
as it can be seen by taking~$v:=1$ in~\eqref{LAWEAKXA}.
\medskip

We now describe in detail the results
stated and proved in this paper.

\subsection{Eigenvalue and eigenfunctions
for the $(\alpha,\beta)$-Neumann condition}

The first set of results that we discuss
here is related to a generalized eigenvalue problem
associated to equation~\eqref{LAWEAKXA-S}
with~$(\alpha,\beta)$-Neumann condition.

Namely, we let~$m:\Omega\to\R$ and
we consider the weighted eigenvalue equation
\begin{equation}\label{probauto}
\begin{dcases}
-\alpha\Delta u +\beta(-\Delta)^su= \lambda mu   & \quad{\mbox{
in }}\;\Omega,
\\
{\mbox{with $(\alpha,\beta)$-Neumann condition.}}\end{dcases}
\end{equation}
According to~\eqref{LAWEAKXA}
the notion of solution in~\eqref{probauto} is in the weak sense
in the space~$X_{\alpha,\beta}$: namely 
we say that~$u\in X_{\alpha,\beta}$ is a solution of~\eqref{probauto}
if
\begin{equation}\label{ORAGSBNmer}\alpha \int_\Omega \nabla u(x) \cdot\nabla v(x)\,dx
+\frac{\beta}{2}
\iint_\Q\frac{(u(x)-u(y))(v(x)-v(y))}{|x-y|^{n+2s}}\,dx\,dy
=\lambda \int_\Omega m(x) u(x)v(x)\,dx,\end{equation}
for every~$v\in X_{\alpha,\beta}$.
\medskip

To deal with the integrability condition of the weight~$m$,
it is convenient to
consider the following ``critical'' exponent:
\begin{equation}\begin{split}\label{qbar}
\underline{q}:=\;&\begin{dcases}
\displaystyle\frac{2^*}{2^*-2} & {\mbox{ if $\beta=0$ and~$n>2$}},\\
\displaystyle\frac{2^*_s}{2^*_s-2} & {\mbox{ if $\beta\ne0$ and~$n>2s$}},\\
1 & {\mbox{ if $\beta=0$ and~$n\le2$, or
if $\beta\ne0$ and~$n\le2s$}},\end{dcases}\\=\;&
\begin{dcases}
\displaystyle\frac{n}{2} & {\mbox{ if $\beta=0$ and~$n>2$}},\\
\displaystyle\frac{n}{2s} & {\mbox{ if $\beta\ne0$ and~$n>2s$}},\\
1 & {\mbox{ if $\beta=0$ and~$n\le2$, or
if $\beta\ne0$ and~$n\le2s$.}}
\end{dcases}
\end{split}\end{equation}
As customary, the exponent~$2^*_s$ denotes the fractional
Sobolev critical exponent for~$n>2s$
and it is equal to~$\frac{2n}{n-2s}$. Similarly,
the exponent~$2^*$ denotes the classical
Sobolev critical exponent for~$n>2$
and it is equal to~$\frac{2n}{n-2}$.

Furthermore, we suppose that
\begin{equation}\label{feuwtywvv123445}
m\in L^q(\Omega),\quad {\mbox{for some $q \in \big(\underline{q},+\infty\big]$,}}
\end{equation}
where~$\underline{q}$ is given in~\eqref{qbar}.

In this setting, problem \eqref{probauto}
admits a spectral decomposition of classical flavor,
according to the following result:

\begin{prop}\label{PROAUTOVA}
Suppose that~$m^+$, $m^-\not\equiv 0$ and\footnote{As customary, we use the standard
notation
$$ m^+(x):=\max\{0,m(x)\}\qquad {\mbox{ and }}\qquad
m^-(x):=\max\{0,-m(x)\}. $$} that
\begin{equation}\label{yt5645867600000}
\int_\Omega m(x)\,dx\neq 0.\end{equation}
Then, problem \eqref{probauto} admits two unbounded sequences of 
eigenvalues:
\[
\cdots\le\lambda_{-2}\leq \lambda_{-1}<\lambda_0=0
<\lambda_1\leq 
\lambda_2 \le\cdots\;\;.
\]
In particular, if 
$$\int_\Omega m(x)\,dx<0,$$ then
\begin{equation}\label{lopouygbv}
\lambda_1=\min_{u\in X_{\alpha,\beta}}
\left\lbrace [u]^2_{X_{\alpha,\beta}}\, {\mbox{ s.t. }}
\int_\Omega m(x)u^2(x)\,dx=1 \right\rbrace
\end{equation}
where we use the notation in~\eqref{seminorm}.
If instead
$$\int_\Omega m(x)\,dx>0,$$ then
\[
\lambda_{-1}=-\min_{u\in X_{\alpha,\beta}}
\left\lbrace [u]^2_{X_{\alpha,\beta}} \,{\mbox{ s.t. }}
\int_\Omega m(x)u^2(x)\,dx=-1 \right\rbrace.
\]
\end{prop}

The first positive eigenvalue $\lambda_1$, as given by
Proposition~\ref{PROAUTOVA}, has the following structural
properties:

\begin{prop}\label{prop:lambda}
Suppose that~$m^+\not\equiv 0$ and
$$\int_{\Omega} m(x)\,dx<0.$$
Then,
the first positive eigenvalue $\lambda_1$ of \eqref{probauto} is 
simple, and the first eigenfunction $e$ can be taken such that~$e\ge0$.

A similar statement holds if $m^-\not\equiv 0$ and
$$ \int_{\Omega} m(x)\,dx>0.$$
\end{prop}

To deal with the eigenvalue problem in~\eqref{probauto},
it is convenient to recall the notation in~\eqref{Xdefab} and
to introduce the space
\begin{equation}\label{defV}
V_m:=\left\lbrace u\in X_{\alpha,\beta}\,{\mbox{ s.t. }}
\int_\Omega m(x)u(x)\,dx=0 \right\rbrace.
\end{equation}
To ease the notation, we will simply write~$V$ instead of~$V_m$ in what follows.
We observe that, in view of~\eqref{ZERAJS-jPP},
\begin{equation}\label{ALLEI}
{\mbox{all the eigenfunctions of problem~\eqref{probauto}
belong to~$V$.}}
\end{equation}
As we will see in Corollary~\ref{G:BOUND},
a global bound holds true for these eigenfunctions.
To obtain this bound, we develop a general theory,
of independent interest, to bound globally from below
the weak subsolutions that fulfill the $(\alpha,\beta)$-Neumann
conditions, as we now discuss in detail.

\subsection{Global uniform bounds for
subsolutions under $(\alpha,\beta)$-Neumann
condition}

We give here an $L^\infty$-result for solutions,
and more general, subsolutions of
equation~\eqref{LAWEAKXA-S} 
under~$(\alpha,\beta)$-Neumann condition.
To apply this bound to the eigenfunctions
of problem~\eqref{probauto}, it is also convenient to allow an additional
linear term in the equation that we take into account.
The result that we have is the following one:

\begin{theorem}\label{OSC55}
Let~$V$ be as in \eqref{defV} and~$\underline{q}$ be as in~\eqref{qbar}.
Let~$q\in \left(\underline{q},+\infty\right)$ and~$c$, $f\in L^q(\Omega)$.
Let~$u\in V$ satisfy
\begin{equation}\label{WEAK-DGSUBSOL}
\begin{split}&
\alpha \int_{\Omega} \nabla u \cdot\nabla { v }\,dx 
+\frac{\beta}{2}
\iint_{{\mathcal{Q}}}\frac{(u(x)-u(y))({ v }(x)-{ v }(y))}{|x-y|^{{{n}}+2s}}\,dx\,dy
\\&\qquad\le
\int_{\Omega}\big(c(x)u(x)+ f(x)\big)\,{ v }(x)\,dx
\end{split}
\end{equation}
for each~${ v }\in X_{\alpha,\beta}$ such that~${ v }\ge0$ in~$\Omega$.

Then, there exists~$C>0$,
depending on~${n}$, $\alpha$, $\beta$,
$q$, $\Omega$, $\|c\|_{L^q(\Omega)}$ and~$m$ such that
\begin{equation}\label{BOU-o1}
\sup_{\Omega} u^+\le
C\,\left(
\| u^+\|_{L^2(\Omega)}+\|f\|_{L^q(\Omega)}
\right).
\end{equation}
\end{theorem}

In a forthcoming paper, we plan to use Theorem~\ref{OSC55}
as the cornerstone for a regularity theory
for mixed equations under $(\alpha,\beta)$-Neumann conditions.\medskip

As a consequence of~\eqref{ALLEI}
and Theorem~\ref{OSC55} (applied with~$f:=0$ and~$c:=\lambda m$),
we easily obtain the following global bound for eigenfunctions:

\begin{cor}\label{G:BOUND}
All the eigenfunctions of problem~\eqref{probauto}
belong to~$L^\infty(\Omega)$.\end{cor}

In the rest of the paper, we provide full
detailed proofs for
Propositions~\ref{PROAUTOVA}
and~\ref{prop:lambda} (in Section~\ref{AUTOPER})
and for
Theorem~\ref{OSC55} (in Section~\ref{KM:09009936523846765}).

\section{Eigenvalues and eigenfunctions and proof of Propositions~\ref{PROAUTOVA}
and~\ref{prop:lambda}}\label{AUTOPER}

The proofs of Propositions~\ref{PROAUTOVA}
and~\ref{prop:lambda} rely on classical functional analysis,
revisited
in a mixed local-nonlocal framework.
We start these arguments
by pointing out that a Poincar\'e-type inequality holds in the space~$V$
introduced in~\eqref{defV}:

\begin{lem}\label{POI66}
Let~$m$ be such that
\begin{equation}\label{yt5645867600000PRE}
\int_\Omega m(x)\,dx\neq 0.\end{equation}
Then, recalling the notation in~\eqref{seminorm}, we have that
\begin{equation}\label{poincare}
\int_\Omega u^2(x)\, dx\leq C [u]^2_{X_{\alpha,\beta}},
\end{equation} 
for every $u\in V$,
where $C>0$ depends only on~$n$, $\Omega$, $s$ and~$m$. 
\end{lem}

\begin{proof}
We argue
by contradiction and we suppose that
there exists a sequence of functions~$u_k\in V$ such that
\begin{equation}\label{48450y95uyhjr}
\int_\Omega u_k^2(x)\,dx=1
\end{equation}
and
\begin{equation}\label{poinc1}
[u_k]^2_{X_{\alpha,\beta}}<\frac{1}{k}.
\end{equation}
In particular, the sequence $(u_k)_k$ is bounded in 
$X_{\alpha,\beta}$ uniformly in~$k$. As a consequence, from the compact
embedding of 
$X_{\alpha,\beta}$ in $L^2(\Omega)$ (see e.g. Corollary~7.2
in~\cite{MR2944369} if~$\alpha=0$), we have that,
up to a subsequence, $u_k$ converges to some function~$u\in L^2(\Omega)$
as~$k\to+\infty$. Moreover, $u_k$ converges to~$u$
a.e. in $\Omega$ as~$k\to+\infty$, and~$|u_k|\le h$ for some~$h\in L^2(\Omega)$
for every~$k\in\N$ (see e.g. Theorem~IV.9 in~\cite{MR697382}).

As a result, since $u_k\in V$, we can apply the Dominated Convergence Theorem
to conclude that
\begin{equation}\label{poinc2}
\int_\Omega m(x)u(x)\,dx=0.
\end{equation}
In addition, we deduce from~\eqref{48450y95uyhjr} that
\begin{equation}\label{48450y95uyhjrbis}
\int_\Omega u^2(x)\,dx=1.\end{equation}

On the other hand, by the Fatou Lemma, the lower semicontinuity
of the~$L^2$-norm and~\eqref{poinc1} we have that
\begin{equation}\begin{split}\label{45wejyuiniunliol}
&\frac{\alpha}2\int_{\Omega}|\nabla u|^2\,dx
+\frac\beta4\int_\Omega\int_\Omega\frac{|u(x)-u(y)|^2}{|x-y|^{n+2s}}\,dx\,dy
\\&\qquad \le \liminf_{k\to+\infty}\left(
\frac{\alpha}2\int_{\Omega}|\nabla u_k|^2\,dx+\frac{\beta}2
\int_\Omega\int_\Omega\frac{|u_k(x)-u_k(y)|^2}{|x-y|^{n+2s}}\,dx\,dy\right)
\le  \lim_{k\to+\infty}\frac1k=0.
\end{split}\end{equation}
Now, if~$\beta=0$, this says that
$$ \int_{\Omega}|\nabla u|^2\,dx=0,$$
which implies that~$u$ is constant in~$\Omega$, thanks to~\eqref{connected}.
If instead~$\beta\ne0$, we have
from~\eqref{45wejyuiniunliol} that
$$ \int_\Omega\int_\Omega\frac{|u(x)-u(y)|^2}{|x-y|^{n+2s}}\,dx\,dy=0,$$
which gives that~$u$ is constant in $\Omega$.
Hence in both case, we have that~$u$ is constant in~$\Omega$.

Moreover, we observe that~$u$ cannot vanish identically in~$\Omega$, in light
of~\eqref{48450y95uyhjrbis}. Using these observations into~\eqref{poinc2}
we conclude that
$$ \int_\Omega m(x)\,dx=0,$$
which is in contradiction with~\eqref{yt5645867600000PRE}. 
This completes the proof of
formula~\eqref{poincare}.
\end{proof}

We notice that, thanks to~\eqref{poincare}, the seminorm
in~\eqref{seminorm} is actually a norm on the space~$V$ and
it is equivalent to the norm on~$X_{\alpha,\beta}$
given by~\eqref{scalar}.
Moreover, the scalar product defined as
\begin{equation}\label{uf4t}
\langle u,v\rangle_{X_{\alpha,\beta}}:=\alpha \int_\Omega \nabla u\cdot\nabla v\,dx
+\frac\beta2 \iint_\Q\frac{(u(x)-u(y))(v(x)-v(y))}{|x-y|^{n+2s}}\,dx\,dy \end{equation}
is equivalent to the one in $X_{\alpha,\beta}$
given by~\eqref{scalar}. In this setting, we also denote
$$ \|u\|_V:=\sqrt{\langle u,v\rangle_{X_{\alpha,\beta}}}.$$

To complete the functional setting for the eigenvalue problem
in~\eqref{probauto}, we also remark that~$V$ is closed with respect to the weak
convergence:

\begin{lem}\label{lemmaclosed}
The space~$V$ introduced in~\eqref{defV} is closed with respect
to the weak convergence in~$V$.
\end{lem}

\begin{proof}
We take a sequence of functions~$u_j\in V$ weakly converging to some~$u$,
and we claim that~$u\in V$. Indeed, we have that~$u_j$ weakly
converges to~$u$ in~$X_{\alpha,\beta}$,
and~$u\in X_{\alpha,\beta}$. Furthermore, 
by the compact embeddings  (see e.g. Corollary~7.2
in~\cite{MR2944369} if~$\alpha=0$),
$u_j \to u$ in~$L^p(\Omega)$ for any~$p\in [1,2^*_s)$ if~$\alpha=0$
and for any~$p\in [1,2^*)$ if~$\alpha\ne0$.
Moreover, $u_j$ converges to~$u$
a.e. in $\Omega$, and~$|u_j|\le h$ for some~$h\in L^p(\Omega)$
(see e.g. Theorem~IV.9 in~\cite{MR697382}).
As a result, since $u_j\in V$, recalling~\eqref{feuwtywvv123445},
we can apply the Dominated Convergence Theorem
to conclude that
\begin{equation*}
\int_\Omega m(x)u(x)\,dx=0,
\end{equation*}
which proves that~$u\in V$, thus completing the proof of Lemma~\ref{lemmaclosed}. 
\end{proof}

With this preliminary work, we can give the proofs
of Propositions~\ref{PROAUTOVA} and~\ref{prop:lambda}
by relying on functional analysis methods:

\begin{proof}[Proof of Proposition~\ref{PROAUTOVA}]
We notice that 
\begin{equation}\label{notice}
{\mbox{the simple eigenfunction $\lambda_0=0$ has only 
constant functions as eigenfunctions.}}\end{equation}
Indeed, if~$u$ is an eigenfunction
associated to~$\lambda_0=0$, then, by~\eqref{ORAGSBNmer},
\begin{equation}\label{jietyugvsde3957}
\alpha \int_\Omega \nabla u(x) \cdot\nabla v(x)\,dx 
+\frac{\beta}{2}
\iint_\Q\frac{(u(x)-u(y))(v(x)-v(y))}{|x-y|^{n+2s}}\,dx\,dy=0,
\end{equation}
for all functions~$v\in X_{\alpha,\beta}$.
In particular, taking~$u$ as test function in~\eqref{jietyugvsde3957}, 
we obtain that
\begin{equation}\label{jietyugvsde395722}
\alpha \int_\Omega |\nabla u(x)|^2\,dx 
+\frac{\beta}{2}
\iint_\Q\frac{|u(x)-u(y)|^2}{|x-y|^{n+2s}}\,dx\,dy=0.
\end{equation}
Now, if~$\beta=0$, formula~\eqref{jietyugvsde395722} implies that
$$ \int_\Omega |\nabla u(x)|^2\,dx =0.$$
This, together with~\eqref{connected}, gives that~$u$ is constant in~$\Omega$,
thus proving~\eqref{notice} in this case.

If instead~$\beta\neq0$, we deduce from~\eqref{jietyugvsde395722} that
$$ \iint_\Q\frac{|u(x)-u(y)|^2}{|x-y|^{n+2s}}\,dx\,dy=0,$$
which implies~\eqref{notice}.

Now, to obtain the other 
eigenvalues, we restrict to the space~$V$ introduced in~\eqref{defV}.
We point out that the assumption in~\eqref{yt5645867600000}
guarantees that the Poincar\`e inequality in~\eqref{poincare}
holds true on the space~$V$.

Also, we define the linear operator $T:V\to V$ by 
\begin{equation}\label{defT}
\langle Tv,w\rangle_{X_{\alpha,\beta}}=\int_\Omega m(x)v(x)w(x)\,dx,
\end{equation}
for every~$v$, $w\in V$.

It is easy to see that~$T$ is symmetric. Furthermore,
we claim that
\begin{equation}\label{compact}
{\mbox{$T$ is compact.}}\end{equation}
To prove this, we let $(u_j)_j$ be a bounded sequence in $V$. 
Then, $(u_j)_j$ is a bounded sequence in~$X_{\alpha,\beta}$, and therefore there exists~$u\in
X_{\alpha,\beta}$ such that~$u_j$ weakly converges to~$u$ in~$X_{\alpha,\beta}$
as~$j\to+\infty$. Moreover, from Lemma~\ref{lemmaclosed}, we have that~$u\in V$.

Now, by the compact embeddings,
\begin{equation}\label{poit78676}
{\mbox{$u_j \to u$ in $L^p(\Omega)$ for any~$p\in [1,2^*_s)$ if~$\alpha=0$
and for any~$p\in [1,2^*)$ if~$\alpha\ne0$.}}\end{equation}

Using \eqref{defT} with $v:=u_j-u$ and $w:=Tu_j-Tu$, 
we deduce that
\begin{equation}\label{r435fnasdaw25}
\|Tu_j-Tu\|_V^2=\langle T(u_j-u), Tu_j-T_u\rangle_{X_{\alpha,\beta}}
=\int_\Omega m(u_j-u)\big(Tu_j-Tu\big)\,dx.\end{equation}
Now we apply
H\"older's inequality with exponents~$q$, as given in~\eqref{feuwtywvv123445},
$p$, as given by~\eqref{poit78676}, and either~$2^*_s$ if~$\alpha=0$
or~$2^*$ if~$\alpha\ne0$. In this way, using also the continuous embedding
of~$V$ either in~$L^{2^*_s}(\Omega)$ if~$\alpha=0$
or~$L^{2^*}(\Omega)$ if~$\alpha\ne0$,
we obtain from~\eqref{r435fnasdaw25}
that
$$ \|Tu_j-Tu\|_V^2
\leq C\|m\|_{L^q(\Omega)}\|u_j-u\|_{L^p(\Omega)}
\|Tu_j-Tu\|_V,
$$
for some positive constant~$C$ independent of~$j$.
This implies that
\[
\|Tu_j-Tu\|_V\leq  C \|m\|_{L^q(\Omega)}\|u_j-u\|_{L^p(\Omega)}.
\]
Accordingly, recalling~\eqref{poit78676}, we obtain
that~$Tu_j\to Tu$ in $V$ as~$j\to+\infty$.
This completes the proof of~\eqref{compact}.

Now we observe that,
in light of~\eqref{ORAGSBNmer}, and recalling~\eqref{uf4t} and~\eqref{defT},
we can write the weak formulation of problem \eqref{probauto} as
\begin{equation}\label{SJNDI-32i3rtjrgnnvnbn}
\langle u,v\rangle_{X_{\alpha,\beta}} =\lambda \langle Tu,v\rangle_{X_{\alpha,\beta}}
 \quad {\mbox{ for all }} v\in X_{\alpha,\beta}.
\end{equation}
Therefore, we can apply standard results in spectral theory of 
self-adjoint and compact operators to obtain the existence and the 
variational characterization of eigenvalues (see 
e.g.~\cite[Propo\-si\-tion 1.10]{defi}; see also~\cite{MR576277}
and the references therein for related classical results).
\end{proof}

\begin{proof}[Proof of Proposition~\ref{prop:lambda}]
We first observe that if~$\beta\ne0$ and~$w$ is an eigenfunction
according to~\eqref{probauto}, then
\begin{equation}\label{ENOUGH}
{\mbox{$w\equiv0$ in~$\Omega$ entails that~$w\equiv0$
in the whole of~$\R^n$.}}
\end{equation}
To check this, suppose that~$w\equiv0$ in~$\Omega$
and write~\eqref{probauto}
explicitly as in~\eqref{ORAGSBNmer}, namely
\begin{equation}\label{WEAKSOL-mla}
\begin{split}&
\alpha \int_\Omega \nabla w(x) \cdot\nabla v(x)\,dx 
+\frac{\beta}{2}
\iint_\Q\frac{(w(x)-w(y))(v(x)-v(y))}{|x-y|^{n+2s}}\,dx\,dy
\\&\qquad=\lambda
\int_\Omega m(x)\,w(x)\,v(x)\,dx
\end{split}
\end{equation}
for all functions~$v\in X_{\alpha,\beta}$.
In particular, choosing~$v:=w$ in~\eqref{WEAKSOL-mla},
$$ 0=
\frac{\beta}{2}
\iint_\Q\frac{(w(x)-w(y))^2}{|x-y|^{n+2s}}\,dx\,dy=
\beta
\iint_{\Omega\times(\R^n\setminus\Omega)}
\frac{w^2(y)}{|x-y|^{n+2s}}\,dx\,dy.$$
Whence, if~$\beta\ne0$, it follows that~$w(y)=0$ for each~$y\in\Omega$,
thus establishing~\eqref{ENOUGH}.

Now, we prove that
\begin{equation}\label{posi22}
{\mbox{all the eigenfunctions corresponding to~$\lambda_1$ do not change
sign.}}
\end{equation}
For this, we let~$u$ be an eigenfunction corresponding to
the first positive eigenvalue $\lambda_1$.
In particular, recalling~\eqref{lopouygbv},
we have that~$u\in X_{\alpha,\beta}$ and
\begin{equation}\label{forse}
\int_{\Omega } m(x)u^2(x)\,dx=1.\end{equation}
If~$u$ is either nonnegative or nonpositive, then~\eqref{posi22} is established.
Hence, we are left with the case in which~$u$ changes sign in~$\Omega$.
In this case, we have that both~$u^+\not\equiv0$
and~$u^-\not\equiv 0$, and we claim that
\begin{equation}\label{posi33}
{\mbox{both~$u^+$ and $u^-$ are
eigenfunctions corresponding to~$\lambda_1$.}}
\end{equation}
To this end,
we notice that
\begin{equation}\label{49vbhgjhb}
\int_{\Omega} u^2(x)\,dx =\int_\Omega (u^+(x))^2\,dx + \int_\Omega (u^-(x))^2\,dx.
\end{equation}
Moreover, recalling~\eqref{seminorm}, by inspection one sees that
\begin{equation}\begin{split}\label{iehtierhgg}&
[u]_{X_{\alpha,\beta}}^2\\=\;& \alpha \int_\Omega |\nabla u|^2\,dx
+\beta \iint_\Q\frac{|u(x)-u(y)|^2}{|x-y|^{n+2s}}\,dx\,dy\\
=\;&\alpha \int_\Omega \left(|\nabla u^+|^2+|\nabla u^-|^2\right)\,dx
+\frac\beta2 \iint_\Q\frac{|u^+(x)-u^+(y)|^2}{|x-y|^{n+2s}}\,dx\,dy\\
&\qquad + \frac\beta2 \iint_\Q\frac{|u^-(x)-u^-(y)|^2}{|x-y|^{n+2s}}\,dx\,dy
- \beta \iint_\Q\frac{(u^+(x)-u^+(y))(u^-(x)-u^-(y))}{|x-y|^{n+2s}}\,dx\,dy
\\ \geq\;&[u^+]_{X_{\alpha,\beta}}^2+[u^-]_{X_{\alpha,\beta}}^2.
\end{split}\end{equation}
This and~\eqref{49vbhgjhb} imply that~$u^+$, $u^-\in X_{\alpha,\beta}$.

Also, in light of~\eqref{forse}, we have that
\[ 1=
\int_\Omega m(x)u^2(x)\,dx=\int_\Omega m(x)(u^+(x))^2\,dx
+\int_\Omega m(x)(u^-(x))^2\,dx.
\]
Hence, using this and~\eqref{iehtierhgg}, and
recalling the characterization of~$\lambda_1$ given in~\eqref{lopouygbv}, 
\begin{equation}\label{disug+-}
\frac{1}{\lambda_1}=\frac{1}{[u]_{X_{\alpha,\beta}}^2}=
\frac{\displaystyle\int_\Omega m(x)u^2(x)\,dx}{[u]_{X_{\alpha,\beta}}^2}
\leq \frac{\displaystyle\int_\Omega m(x)(u^+(x))^2\,dx
+\int_\Omega m(x)(u^-(x))^2\,dx}
{[u^+]_{X_{\alpha,\beta}}^2+[u^-]_{X_{\alpha,\beta}}^2}.
\end{equation}

Now we claim that, for any $a_1$, $a_2$, $b_1$, $b_2>0$, either
\begin{equation}\label{ab1}
\frac{a_1+a_2}{b_1+b_2}=\frac{a_1}{b_1}=\frac{a_2}{b_2},
\end{equation}
or
\begin{equation}\label{ab2}
\frac{a_1+a_2}{b_1+b_2}<\max\left\lbrace\frac{a_1}{b_1},
\frac{a_2}{b_2}\right\rbrace.
\end{equation}
Indeed, if $\frac{a_1}{b_1}=\frac{a_2}{b_2}$, then 
$$ \frac{a_1+a_2}{b_1+b_2}=\frac{a_2}{b_2}\cdot \frac{\frac{a_1}{a_2}+1}{\frac{b_1}{b_2}+1}=
\frac{a_2}{b_2}\cdot\frac{\frac{a_1}{a_2}+1}{\frac{a_1}{a_2}+1}=\frac{a_2}{b_2},$$
that is~\eqref{ab1}. If instead we suppose that 
$\frac{a_1}{b_1}>\frac{a_2}{b_2}$ (being the case in which $\frac{a_1}{b_1}<\frac{a_2}{b_2}$
similar), then
\[
\frac{a_1+a_2}{b_1+b_2}=\frac{b_1(a_1+a_2)}{b_1(b_1+b_2)}
<\frac{a_1b_1+a_1b_2}{b_1(b_1+b_2)}=\frac{a_1(b_1+b_2)}{b_1(b_1+b_2)}=
\frac{a_1}{b_1},
\]
which proves~\eqref{ab2}.

Now, if we suppose that
$$  \frac{\displaystyle\int_\Omega m(x)(u^+(x))^2\,dx}{[u^+]_{X_{\alpha,\beta}}^2}>
\frac{\displaystyle\int_\Omega m(x)(u^-(x))^2\,dx}{[u^-]_{X_{\alpha,\beta}}^2}$$ 
then we deduce from~\eqref{disug+-} and~\eqref{ab2},
applied here with
\begin{eqnarray*}
&& a_1:= \int_\Omega m(x)(u^+(x))^2\,dx, \quad 
a_2:=\int_\Omega m(x)(u^-(x))^2\,dx, \\&& b_1:=[u^+]_{X_{\alpha,\beta}}^2,\quad
{\mbox{ and }}\quad b_2:=[u^-]_{X_{\alpha,\beta}}^2,\end{eqnarray*}
that
$$\frac{1}{\lambda_1} < \frac{\displaystyle\int_\Omega m(x)(u^+(x))^2\,dx}{[u^+]_{X_{\alpha,\beta}}^2},$$
which contradicts the minimality of~$\lambda_1$.
Similarly, if
$$  \frac{\displaystyle\int_\Omega m(x)(u^+(x))^2\,dx}{[u^+]_{X_{\alpha,\beta}}^2}<
\frac{\displaystyle\int_\Omega m(x)(u^-(x))^2\,dx}{[u^-]_{X_{\alpha,\beta}}^2},$$ then
$$\frac{1}{\lambda_1} < \frac{\displaystyle\int_\Omega m(x)(u^-(x))^2\,dx}{[u^-]_{X_{\alpha,\beta}}^2},$$ 
which is again a contradiction with the minimality of~$\lambda_1$.

As a consequence, we have that
$$ \frac{\displaystyle\int_\Omega m(x)(u^+(x))^2\,dx}{[u^+]_{X_{\alpha,\beta}}^2}=
\frac{\displaystyle\int_\Omega m(x)(u^-(x))^2\,dx}{[u^-]_{X_{\alpha,\beta}}^2}.$$
In this case, we can apply~\eqref{ab1} and we obtain from~\eqref{disug+-} that
$$\frac{1}{\lambda_1} \le \frac{\displaystyle\int_\Omega m(x)(u^+(x))^2\,dx}{[u^+]_{X_{\alpha,\beta}}^2}
=\frac{\displaystyle\int_\Omega m(x)(u^-(x))^2\,dx}{[u^-]_{X_{\alpha,\beta}}^2},$$
that is
\begin{equation}\label{-678686uhi}
\lambda_1\ge  \frac{[u^+]_{X_{\alpha,\beta}}^2}{\displaystyle\int_\Omega m(x)(u^+(x))^2\,dx}
=\frac{[u^-]_{X_{\alpha,\beta}}^2}{\displaystyle\int_\Omega m(x)(u^-(x))^2\,dx}.
\end{equation}
Now, if the inequality in~\eqref{-678686uhi}
is strict, we have a contradiction 
with the minimality of~$\lambda_1$. Accordingly,
$$
\lambda_1=  \frac{[u^+]_{X_{\alpha,\beta}}^2}{\displaystyle\int_\Omega m(x)(u^+(x))^2\,dx}
=\frac{[u^-]_{X_{\alpha,\beta}}^2}{\displaystyle\int_\Omega m(x)(u^-(x))^2\,dx}.
$$
This implies that~$u^+$ and $u^-$ are both
eigenfunctions corresponding to~$\lambda_1$ (unless they are trivial)
thus establishing~\eqref{posi33}.

Our next claim is to prove that 
\begin{equation}\label{posi44}
{\mbox{either~$u\equiv u^+$ or~$u\equiv u^-$.}}\end{equation}
We observe that, if~$\beta=0$, then~\eqref{posi44} follows
from the standard maximum principle for the Laplace operator
(see e.g.~\cite{MR2597943}).

If instead~$\beta\ne0$,
we use~\eqref{posi33} and~\eqref{disug+-}
to see that
\begin{equation*}
\frac{1}{\lambda_1}
\leq \frac{\displaystyle\int_\Omega m(x)(u^+(x))^2\,dx
+\int_\Omega m(x)(u^-(x))^2\,dx}
{[u^+]_{X_{\alpha,\beta}}^2+[u^-]_{X_{\alpha,\beta}}^2}
=\frac{1}{\lambda_1}.
\end{equation*}
In particular, equality holds in the latter formula, and accordingly,
recalling~\eqref{iehtierhgg}, we have that
$$ 0=-\iint_\Q\frac{(u^+(x)-u^+(y))(u^-(x)-u^-(y))}{|x-y|^{n+2s}}\,dx\,dy
=\iint_\Q\frac{2u^+(x)u^-(y)}{|x-y|^{n+2s}}\,dx\,dy.
$$
This gives that
\begin{equation}\label{setuerghdfjbv}
u^+(x)u^-(y) =0\qquad {\mbox{ for all }} (x,y)\in\Q.
\end{equation}
We can also suppose that~$u^+\not\equiv0$ (in~$\R^n$ if~$\beta\ne0$
and in~$\Omega$ if~$\beta=0$),
otherwise~$u\equiv u^-$ and we are done. 
This and~\eqref{ENOUGH} give that~$u^+\not\equiv0$ in~$\Omega$.
Hence, we can take~$\bar{x}\in\Omega$ such that~$u^+(\bar{x})\ne0$.
{F}rom this and~\eqref{setuerghdfjbv}, we obtain that
\begin{equation*}
u^+(\bar{x})u^-(y) =0\qquad {\mbox{ for all }} y\in\R^n.
\end{equation*} 
As a consequence, we find that~$u^-\equiv0$ in~$\R^n$,
which establishes~\eqref{posi44}.

In turn, the claim in~\eqref{posi44} implies
the one in~\eqref{posi22}, as desired.

We now prove that~$\lambda_1$ is simple.
First we show that
\begin{equation}\label{GMP}
{\mbox{the geometric multiplicity of 
$\lambda_1$ is 1.}}\end{equation} For this, let $u_1$ and $u_2$ be eigenfunctions 
corresponding to $\lambda_1$. {F}rom~\eqref{posi22} we know that~$u_2$
does not change sign, hence (up to exchanging $u_2$ with~$-u_2$),
we can suppose that~$u_2\ge0$ (in~$\R^n$, if~$\beta\ne0$,
and in~$\Omega$, if~$\beta=0$).

{F}rom this and~\eqref{ENOUGH}, it follows that
$$ \int_\Omega u_2(x)\,dx>0.$$
As a result, we can define
$$ a:=\frac{\displaystyle\int_\Omega u_1(x)\,dx}{
\displaystyle\int_\Omega u_2(x)\,dx},$$
and we find that
\begin{equation}\label{09qwr8hrhtg} \int_\Omega \big(u_1(x)-au_2(x)\big)\,dx=0.\end{equation}
In addition, from~\eqref{posi22}, we know that the eigenfunction~$u_1-au_2$
does not change sign, and therefore~\eqref{09qwr8hrhtg}
entails that~$u_1-au_2\equiv0$ in~$\Omega$.
This and~\eqref{ENOUGH} show that~$u_1-au_2\equiv0$
also in~$\R^n$ when~$\beta\ne0$, and this proves that~$u_1$
and $u_2$ are linearly dependent, giving~\eqref{GMP},
as desired.

Finally, we prove that 
\begin{equation}\label{AGM}
{\mbox{the algebraic multiplicity of 
$\lambda_1$ is 1.}}\end{equation}
To this end, we recall the notation in~\eqref{defV} and~\eqref{defT},
and we claim that 
\begin{equation}\label{TKEP}
{\rm Ker}\big( (I-\lambda_1 T)^2\big)={\rm Ker}(I-\lambda_1 T),\end{equation}
where~$I$ is the identity in~$V$.

To prove~\eqref{TKEP},
let $u\in {\rm Ker}\big((I-\lambda_1 T)^2\big)$. Then,
setting $U:=u-\lambda_1 Tu$,
we have that~$U-\lambda_1 TU=0$, and accordingly,
by~\eqref{SJNDI-32i3rtjrgnnvnbn}, $U$ is an eigenfunction
corresponding to~$\lambda_1$.

{F}rom this fact and~\eqref{GMP},
we conclude that~$U=te_1$ for
some $t\in \R$, where~$e_1$ is a given
eigenfunction corresponding to $\lambda_1$.

As a result,
\[ t\langle e_1,e_1\rangle_{X_{\alpha,\beta}}=\langle U,e_1\rangle_{X_{\alpha,\beta}}=
\langle u-\lambda_1 Tu,e_1\rangle_{X_{\alpha,\beta}}
= \langle u,e_1-\lambda_1 Te_1\rangle_{X_{\alpha,\beta}}= \langle u,0\rangle_{X_{\alpha,\beta}}=0,
\]
which implies that~$t=0$. 
This yields that~$U=0$ and therefore~$u\in
 {\rm Ker}(I-\lambda_1 T)$. This shows that~$
{\rm Ker}\big( (I-\lambda_1 T)^2\big)\subseteq{\rm Ker}(I-\lambda_1 T)$,
and the other inclusion is obvious.

The proof of~\eqref{TKEP} is therefore complete.
{F}rom~\eqref{TKEP}, we obtain that for all~$k\in\N$ with~$k\ge1$,
$$ 
{\rm Ker}\big( (I-\lambda_1 T)^k\big)=
{\rm Ker}(I-\lambda_1 T),$$
and thus
$$ \bigcup_{k=1}^{+\infty}
{\rm Ker}\big( (I-\lambda_1 T)^k\big)=
{\rm Ker}(I-\lambda_1 T).$$
The latter has dimension~1, thanks to~\eqref{GMP},
and therefore the claim in~\eqref{AGM} is established.
\end{proof}

\section{Boundedness of weak subsolutions and proof
of Theorem~\ref{OSC55}}\label{KM:09009936523846765}

For the proof of Theorem~\ref{OSC55},
we give here a general Sobolev inequality for the functions in the space~$V$
introduced in~\eqref{defV}
which can be seen as a natural counterpart of the Poincar\'e
inequality given in
Lemma~\ref{POI66} (the proof is somewhat of classical flavor,
but we provide full details for the sake of completeness):

\begin{lem}\label{NEWSOB}
Let~$m$ be such that
\begin{equation*}
\int_{\Omega} m(x)\,dx\neq0.
\end{equation*}
Let~$\eta$ be the fractional Sobolev
exponent~$2^*_s:=\frac{2n}{n-2s}$ if~$\beta\ne0$ and~$n>2s$, the
classical Sobolev exponent~$2^*:=\frac{2n}{n-2}$ if~$\beta=0$ and~$n>2$
and~$\eta\ge1$
arbitrary in the other cases.

If~$V$ is as in \eqref{defV} and~$u\in V$, then
\begin{equation}\label{D-SON}
\int_{\Omega} u^\eta(x)\,dx\le C\,
\left( \alpha \,\int_{\Omega} |\nabla u(x)|^2\,dx
+\frac{\beta}{2}\,\iint_{{\mathcal{Q}}}\frac{(u(x)-u(y))^2}{|x-y|^{{{n}}+2s}}\,dx\,dy
\right)^{\frac\eta2},
\end{equation}
where~$C>0$ depends only on~$n$, $\Omega$, $s$ and~$m$.
\end{lem}

\begin{proof} As usual, in this proof we will freely
rename~$C>0$ line after line. First of all, we observe
that the following ``generalized'' Sobolev inequality for
any function~$f\in X_{\alpha,\beta}$ holds true:
\begin{equation}\label{GENESOB-1}
\| f\|_{L^{{\eta_1}}(\Omega)}\le C\,\|f\|_{H^1(\Omega)},
\end{equation}
where~${\eta_1}:=2^*$ if~$n>2$, and~${\eta_1}\ge1$
arbitrary if~$n\le2$. Indeed, when~$n>2$, the claim in~\eqref{GENESOB-1}
is the standard Sobolev embedding (see e.g. Theorem~2
on page~279 of~\cite{MR2597943}).
If instead~$n=2$, we let~$\sigma:=\frac{{\eta_1}}{{\eta_1}+1}\in(0,1)$.
By Proposition~2.2 in~\cite{MR2944369}, we know that
\begin{equation}\label{bfSKDcnv}
\|f\|_{H^\sigma(\Omega)}
\le C\,\|f\|_{H^1(\Omega)}.\end{equation}
Also, we have that~$2\sigma <2=n$ and
$$ 2_\sigma^*=\frac{2n}{n-2\sigma}=\frac{2}{1-\sigma}=
2({\eta_1}+1)\ge{\eta_1}.
$$
Hence, by Theorem~6.7 in~\cite{MR2944369},
we obtain that~$\|f\|_{L^{{\eta_1}}(\Omega)}\le C\|f\|_{H^\sigma(\Omega)}$.
{F}rom this and~\eqref{bfSKDcnv}, we obtain~\eqref{GENESOB-1}
in this case.

Finally, when~$n=1$, we have that~\eqref{GENESOB-1}
is a consequence of Morrey embedding (see e.g. Theorem~5
on page~283 of~\cite{MR2597943}). These considerations
complete the proof of~\eqref{GENESOB-1}.

As a fractional counterpart of~\eqref{GENESOB-1}, we notice that
\begin{equation}\label{GENESOB-2}
\| f\|_{L^{{\eta_s}}(\Omega)}\le C\,\|f\|_{H^s(\Omega)},
\end{equation}
where~${\eta_s}:=2^*_s$ if~$n>2s$, and~${\eta_s}\ge1$
arbitrary if~$n\le2s$.
Indeed, when~$n>2s$, we can use Theorem~6.7
in~\cite{MR2944369} and obtain~\eqref{GENESOB-2}.
If instead~$n\le2s$, the claim in~\eqref{GENESOB-2}
is contained in
Theorem~6.10 of~\cite{MR2944369}.

Now we take~$\eta$ as in the statement of
Lemma~\ref{NEWSOB}
and we claim that
\begin{equation}\label{SMDC}
\| f\|_{L^{{\eta}}(\Omega)}\le C\,\big(
\alpha\|f\|_{H^1(\Omega)}+\beta\|f\|_{H^s(\Omega)}\big).
\end{equation}
Indeed,
if~$\beta\ne0$, 
the claim in~\eqref{SMDC} follows from~\eqref{GENESOB-2}.
If instead~$\beta=0$, then necessarily~$\alpha>0$
and thus
the claim in~\eqref{SMDC} is a consequence of~\eqref{GENESOB-1}.

Having proved~\eqref{SMDC}, we can now combine it
with the Poincar\'e inequality in Lemma~\ref{POI66}
in order to complete the proof of~\eqref{D-SON}.
To this end, since~$u\in V$,
Lemma~\ref{POI66} gives that
\begin{equation} \label{JOSDHKFBzv98348ty03rug1}\| u\|_{L^2(\Omega)}\le C\,[u]_{X_{\alpha,\beta}}=
C\,\sqrt{
\frac\alpha2 \,\int_{\Omega} |\nabla u(x)|^2\,dx
+\frac{\beta}{4}\,\iint_{{\mathcal{Q}}}\frac{(u(x)-u(y))^2}{|x-y|^{{{n}}+2s}}\,dx\,dy
}.\end{equation}
Moreover, by~\eqref{SMDC},
\begin{equation}\label{JOSDHKFBzv98348ty03rug}
\begin{split}
\| u\|_{L^{{\eta}}(\Omega)}\,&\le C\,\big(
\alpha\|u\|_{H^1(\Omega)}+\beta\|u\|_{H^s(\Omega)}\big)\\
&\le C\,\sqrt{
\frac\alpha2 \,\int_{\Omega} |\nabla u(x)|^2\,dx
+\frac{\beta}{4}\,\iint_{{\mathcal{Q}}}\frac{(u(x)-u(y))^2}{|x-y|^{{{n}}+2s}}\,dx\,dy
}+C\,\|u\|_{L^2(\Omega)}.
\end{split}\end{equation}
Then, we insert~\eqref{JOSDHKFBzv98348ty03rug1}
into~\eqref{JOSDHKFBzv98348ty03rug}, and we obtain~\eqref{D-SON},
as desired.
\end{proof}

Now, we dive into the details of the proof
of Theorem~\ref{OSC55}, which is based
on a suitable choice of test functions
and an iteration argument.

\begin{proof}[Proof of Theorem~\ref{OSC55}] We combine for this proof some
classical and nonlocal techniques, see e.g.~\cite{MR1669352, MR1911531, MR3060890, MR3161511, MR3237774, MR3542614, MR3593528}.
Differently from the previous
literature, we focus here on the case of
the $(\alpha,\beta)$-Neumann conditions.
For the facility of the reader, we try to make our arguments
as self-contained as possible.

Given~$k\ge0$, we let~$v:=(u-k)^+$.
We claim that
\begin{equation}\label{MASd-sd}
(u(x)-u(y))( v(x)-v(y))\ge(v(x)-v(y))^2.
\end{equation}
To prove this, we can suppose that~$u(x)\ge u(y)$, up to exchanging the roles of~$x$ and~$y$.
Also, if both~$u(x)$ and~$u(y)$ are larger than~$k$,
we have that~$v(x)=u(x)-k$ and~$v(y)=u(y)-k$, and thus~\eqref{MASd-sd}
follows in this case (in fact, with equality instead of inequality).
Therefore, we can suppose that~$u(x)\ge k\ge u(y)$,
whence~$v(x)=u(x)-k$ and~$v(y)=0$,
and then
\begin{eqnarray*}&&
(u(x)-u(y))( v(x)-v(y))-(v(x)-v(y))^2
= (u(x)-u(y))( u(x)-k)-(u(x)-k)^2\\&&\qquad
= \big( (u(x)-u(y))-(u(x)-k)\big)( u(x)-k)=(k-u(y))(u(x)-k)\ge0.
\end{eqnarray*}
This establishes~\eqref{MASd-sd}.

By~\eqref{MASd-sd},
\begin{equation}\label{19gtasgbcsd-2}
\iint_{{\mathcal{Q}}}\frac{(u(x)-u(y))({ v }(x)-{ v }(y))}{|x-y|^{{{n}}+2s}}\,dx\,dy\ge
\iint_{{\mathcal{Q}}}\frac{(v(x)-v(y))^2}{|x-y|^{{{n}}+2s}}\,dx\,dy.\end{equation}
In addition,
\begin{eqnarray*}
\int_{\Omega} \nabla u(x)\cdot\nabla{ v }(x)\,dx=
\int_{\Omega} |\nabla v(x)|^2\,dx.
\end{eqnarray*}
Consequently, by~\eqref{WEAK-DGSUBSOL},
\begin{equation}\label{0980980987654tgb}
\begin{split}
{\mathcal{I}}
\; :=\;&
\alpha \,\int_{\Omega} |\nabla v(x)|^2\,dx
+\frac{\beta}{2}\,\iint_{{\mathcal{Q}}}\frac{(v(x)-v(y))^2}{|x-y|^{{{n}}+2s}}\,dx\,dy
\\ \le\;&
\int_{\Omega}\big(c(x)u(x)+ f(x)\big)\,{ v }(x)\,dx\\
\le\;&
\int_{\Omega}\Big( |c(x)|\,|u(x)|\,v(x)+ |f(x)|\,v(x)\Big)\,dx.
\end{split}
\end{equation} 
We also remark that
\begin{equation}\label{876janscvd945t}
|u(x)|\,v(x)\le 4(v^2(x)+k^2).
\end{equation}
Indeed, if~$u(x)\le k$, then~$v(x)=0$ and~\eqref{876janscvd945t}
plainly follows. If instead~$u(x)>k$, then~$v(x)=u(x)-k$,
and consequently
\begin{eqnarray*}&& |u(x)|\,v(x)- 4v^2(x)-4k^2=u(x)\,v(x)- 4v^2(x)-4k^2=
(v(x)+k)\,v(x)- 4v^2(x)-4k^2\\&&\qquad=kv(x)-3v^2(x)-4k^2
\le0,\end{eqnarray*}
thus establishing~\eqref{876janscvd945t}.

{F}rom~\eqref{0980980987654tgb} and~\eqref{876janscvd945t},
we conclude that
\begin{equation}\label{0980980987654tgb-22}
{\mathcal{I}}
\le C\,
\int_{\Omega\cap\{v\ne0\}} \Big(
|c(x)|\,v^2(x)+ k^2|c(x)|+
|f(x)|\,v(x)\Big)\,dx ,
\end{equation} 
up to renaming~$C>1$.

Now, we denote by~${\mathcal{Z}}$
the Lebesgue measure of the set~$\Omega\cap\{v\ne0\}=\Omega\cap\{u>k\}$
and we let~$\eta$ be as in the statement of Lemma~\ref{NEWSOB},
with the additional requirement that~$\eta>\frac{2q}{q-1}$ 
if~$\beta\ne0$ and~$n\le2s$, and if~$\beta=0$ and~$n\le2$
(these situations corresponding to ``the other cases''
mentioned in the statement of Lemma~\ref{NEWSOB}).

We claim that
\begin{equation} \label{H66OAK}\frac1{q}+\frac{1}{\eta} < 1.
\end{equation}
Indeed, we use here~\eqref{qbar} and we see that, if~$\beta=0$ and~$n>2$,
$$ \frac1{q}+\frac{1}{\eta}<
\frac1{\underline{q}}+\frac{n-2}{2n}=\frac{2}{n}+\frac{n-2}{2n}
=\frac{n+2}{2n}<1.$$
If instead~$\beta\ne0$ and~$n>2s$,
$$ \frac1{q}+\frac{1}{\eta}<
\frac1{\underline{q}}+\frac{n-2s}{2n}=\frac{2s}{n}+\frac{n-2s}{2n}=\frac{n+2s}{2n}
<1.$$
In all the other cases,
$$ \frac1{q}+\frac{1}{\eta}<
\frac1{q}+\frac{q-1}{q}
=1.$$
These observations prove~\eqref{H66OAK}.

Now, from~\eqref{H66OAK},
we can define
\begin{equation}\label{ETATT}\eta':=\frac{1}{1-\displaystyle\frac1q-\frac1\eta}
\end{equation}
and we can
exploit the H\"older inequality with exponents~$q$
and~$\eta$ and~$\eta'$,  thus finding that
\begin{equation}\label{09oqdwfkkk89PS}
\int_{\Omega} |f(x)|\,v(x)\,dx\le
\|f\|_{L^q(\Omega)}\left(
\int_{\R^{{n}}} (v(x))^\eta
\right)^{\frac1\eta}\,{\mathcal{Z}}^{\frac1{\eta'}}
\end{equation}
We fix now~$\delta\in(0,1)$, to be taken conveniently small in what
follows, 
and we claim that
\begin{equation}\label{KS-dfp}
\int_{\Omega} |f(x)|\,v(x)\,dx\le\delta{\mathcal{I}}
+C_\delta\,
\|f\|_{L^q(\Omega)}^2\,{\mathcal{Z}}^{\vartheta},
\end{equation}
with (recalling~\eqref{ETATT})
\begin{equation}\label{KS-dfp-0983w4}
\vartheta:=\frac2{\eta'}=2\left(1-\frac1q-\frac1\eta\right),\end{equation}
for a suitable~$C_\delta>1$. 

Indeed, 
using~\eqref{09oqdwfkkk89PS} and Lemma~\ref{NEWSOB},
\begin{eqnarray*}
\int_{\Omega} |f(x)|\,v(x)\,dx&\le&C\,
\|f\|_{L^q(\Omega)}\,\sqrt{\mathcal{I}}\,{\mathcal{Z}}^{\frac1{\eta'}}\\
&\le&\delta\,{\mathcal{I}}+C_\delta\,\Big(
\|f\|_{L^q(\Omega)}\,{\mathcal{Z}}^{\frac1{\eta'}}\Big)^2,
\end{eqnarray*}
which gives~\eqref{KS-dfp}.

Then, combining~\eqref{0980980987654tgb-22}
and~\eqref{KS-dfp}, we find that
\begin{equation*}
\begin{split}&
{\mathcal{I}}
\le C\,
\int_{\Omega \cap\{v\ne0\}} \Big(
|c(x)|\,v^2(x)+ k^2|c(x)|\Big)\,dx +
C\delta\,{\mathcal{I}}+C_\delta\,
\|f\|_{L^q(\Omega)}^2\,{\mathcal{Z}}^{\vartheta}\end{split}\end{equation*} 
up to renaming constants.

Consequently, choosing~$\delta$ sufficiently small
(and considering~$\delta$ fixed from now on), we obtain
\begin{equation}\label{CAC-PIV}
\begin{split}&
{\mathcal{I}}
\le C\,
\int_{\Omega\cap\{v\ne0\}} \Big(
|c(x)|\,v^2(x)+ k^2|c(x)|\Big)\,dx +C\,
\|f\|_{L^q(\Omega)}^2{\mathcal{Z}}^{\vartheta}
\end{split}\end{equation} 
up to renaming constants.

In this setting, formula~\eqref{CAC-PIV}
will play a role of a pivotal Caccioppoli-type inequality, according to
the following argument. We claim that there exists~$c_\star>0$
such that if~${\mathcal{Z}}<c_\star$ then
\begin{equation}\label{CAC-PIV-2}
\begin{split}&
{\mathcal{I}}
\le C\,\big(k^2+
\|f\|_{L^q(\Omega)}^2\big)\,{\mathcal{Z}}^{1-\frac1q}
.\end{split}\end{equation} 
To check this, we recall~\eqref{KS-dfp-0983w4},
and we use the H\"older inequality and Lemma~\ref{NEWSOB}
to see that
\begin{eqnarray*}
&&\int_{\Omega} 
|c(x)|\,v^2(x)\,dx
\le\|c\|_{L^q(\Omega)}\,\| v\|_{L^{\eta}(\Omega)}^{2} \;{\mathcal{Z}}^{
1-\frac1q-\frac2\eta}\\&&\qquad\qquad=
\|c\|_{L^q(\Omega)}\,\| v\|_{L^{\eta}(\Omega)}^{2} \;
{\mathcal{Z}}^{\frac\vartheta2-\frac1\eta}\le C\,{\mathcal{I}} \;
{\mathcal{Z}}^{\frac\vartheta2-\frac1\eta}
\end{eqnarray*}
and
\begin{eqnarray*}
\int_{\Omega\cap\{v\ne0\}} |c(x)|\,dx\le\|c\|_{L^q(\Omega)}
\;{\mathcal{Z}}^{1-\frac1{q}}
\le C\,{\mathcal{Z}}^{1-\frac1{q}}.
\end{eqnarray*}
We stress that here the constants denoted by~$C$ are allowed
to depend also on~$\|c\|_{L^q(\Omega)}$.
Plugging this information into~\eqref{CAC-PIV}, we obtain that
\begin{equation*}
\begin{split}&
{\mathcal{I}}
\le 
C\,{\mathcal{I}} \;
{\mathcal{Z}}^{\frac\vartheta2-\frac1\eta} +
C\,k^2\,{\mathcal{Z}}^{1-\frac1{q}}+C\,
\|f\|_{L^q(\Omega)}^2{\mathcal{Z}}^{\vartheta}.
\end{split}\end{equation*} 
Noticing that~$\frac\vartheta2-\frac1\eta>0$ and~$1-\frac1{q}\le\vartheta$,
if~$|\mathcal{Z}|$ is sufficiently small we obtain~\eqref{CAC-PIV-2}, as desired.

We also remark that, by Lemma~\ref{NEWSOB},
$$ \int_{\Omega}v^2(x)\,dx\le\left(
\int_{\Omega}v^\eta(x)\,dx\right)^{\frac2\eta}\,{\mathcal{Z}}^{1-\frac2\eta}\le{\mathcal{I}}\,{\mathcal{Z}}^{1-\frac2\eta}.$$
This and~\eqref{CAC-PIV-2} yield that, if~$|\mathcal{Z}|\le c_\star$,
\begin{equation}\label{CAC-PIV-3}
\begin{split}&
\int_{\Omega}v^2(x)\,dx
\le C\,\big(k^2+
\|f\|_{L^q(\Omega)}^2\big)\,{\mathcal{Z}}^{2-\frac1q-\frac2\eta}
.\end{split}\end{equation} 
We stress that~$2-\frac1q-\frac2\eta>1$,
hence~\eqref{CAC-PIV-3} gives that
\begin{equation}\label{KSD-34ro-1}
\begin{split}&
\int_{\Omega}v^2(x)\,dx
\le C\,\big(k^2+
\|f\|_{L^q(\Omega)}^2\big)\,{\mathcal{Z}}^{1+\epsilon_0}
\end{split}\end{equation} 
for some~$\epsilon_0>0$.

That is, setting~$A(k):=\Omega\cap\{u> k\}$
and
$$\varphi(k):=\int_{A(k)}(u(x)-k)^2(x)\,dx:=\int_{\Omega}v^2(x)\,dx,$$
in light of~\eqref{KSD-34ro-1}
we can write that, if~$|A(k)|\le c_\star$, then
\begin{equation}\label{KSD-34ro-2}
\begin{split}&
\varphi(k)
\le C\,\big(k^2+
\|f\|_{L^q(\Omega)}^2\big)\,|A(k)|^{1+\epsilon_0}.
\end{split}\end{equation} 
We observe that if~$x\in A(k)$ then~$u(x)>k$ and thus~$u^+(x)>k$.
Therefore,
$$ |A(k)|\le \frac{1}{k}\int_{A(k)} u^+(x)\,dx\le
\frac{\sqrt{|A(k)|}}{k}\,\|u^+\|_{L^2(\Omega)}.$$
Hence, it follows that
\begin{equation}\label{COKAPPA-0} |A(k)|\le\left(\frac{\|u^+\|_{L^2(\Omega)}}{k}\right)^2
\le c_\star,\end{equation}
as long as
\begin{equation}\label{COKAPPA}
k\ge \frac{\|u^+\|_{L^2(\Omega)}}{\sqrt{c_\star}}=:\kappa.
\end{equation}
In particular, in view of~\eqref{COKAPPA-0},
we know that~\eqref{KSD-34ro-2} holds true for all~$k$ satisfying~\eqref{COKAPPA}.

Now we define, for every~$\ell\in\N$,
\begin{eqnarray*}
&&K:=\kappa+\|f\|_{L^q(\Omega)}
\\{\mbox{and }} &&k_\ell:= \kappa+K\left(1-\frac1{2^\ell}\right).\end{eqnarray*}
We point out that
$$ k_\ell-k_{\ell-1}=\frac{K}{2^\ell},$$
and, as a result, if~$x\in A(k_\ell)$ then~$u(x)- k_{\ell-1}\ge k_\ell-k_{\ell-1}=\frac{K}{2^\ell}$.

For this reason, we have that 
$$ A(k_\ell)\le\frac{2^{2\ell}}{K^2}\int_{A(k_\ell)} (u(x)-k_{\ell-1})^2\,dx\le
\frac{2^{2\ell}}{K^2}\int_{A(k_{\ell-1})} (u(x)-k_{\ell-1})^2\,dx=\frac{2^{2\ell}}{K^2}\varphi(k_{\ell-1}).$$
Using this information together with~\eqref{KSD-34ro-2} (exploited here with~$k:=k_\ell$,
and we remark that~$k_\ell\ge\kappa$, hence condition~\eqref{COKAPPA}
is satisfied),
we discover that
\begin{equation}\label{KSD-34ro-3}
\begin{split}&
\varphi(k_\ell)
\le \frac{C^\ell\,(k_\ell^2+
\|f\|_{L^q(\Omega)}^2)}{K^2}\,(\varphi(k_{\ell-1}))^{1+\epsilon_0}.
\end{split}\end{equation} 
Since~$k_\ell\le \kappa+K$, up to renaming constants we obtain from~\eqref{KSD-34ro-3}
that
\begin{equation*}
\varphi(k_\ell)
\le \frac{C^\ell\,(\kappa^2+K^2)}{K^2}\,(\varphi(k_{\ell-1}))^{1+\epsilon_0},
\end{equation*} 
and consequently, if~$c_\star$ is sufficiently small,
$$ 0=\lim_{\ell\to+\infty}\varphi(k_\ell)=\varphi(\kappa+K).
$$
As a result, $u^+(x)\le\kappa+K$, whence the claim in~\eqref{BOU-o1}
plainly follows.
\end{proof}


\begin{bibdiv}
\begin{biblist}


\bib{ABATANGELO}{article}{
author={Abatangelo, Nicola},
author={Cozzi, Matteo},
        title = {An elliptic boundary value problem with fractional 
nonlinearity},
      journal = {arXiv e-prints},
        date = {2020},
          eid = {arXiv:2005.09515},
        pages = {arXiv:2005.09515},
archivePrefix = {arXiv},
       eprint = {2005.09515},
       adsurl = {https://ui.adsabs.harvard.edu/abs/2020arXiv200509515A},
      adsnote = {Provided by the SAO/NASA Astrophysics Data System}
}

\bib{MR3912710}{article}{
   author={Affili, Elisa},
   author={Valdinoci, Enrico},
   title={Decay estimates for evolution equations with classical and
   fractional time-derivatives},
   journal={J. Differential Equations},
   volume={266},
   date={2019},
   number={7},
   pages={4027--4060},
   issn={0022-0396},
   review={\MR{3912710}},
   doi={10.1016/j.jde.2018.09.031},
}

\bib{MR3169773}{article}{
   author={Alfaro, Matthieu},
   author={Coville, J\'{e}r\^{o}me},
   author={Raoul, Ga\"{e}l},
   title={Travelling waves in a nonlocal reaction-diffusion equation as a
   model for a population structured by a space variable and a phenotypic
   trait},
   journal={Comm. Partial Differential Equations},
   volume={38},
   date={2013},
   number={12},
   pages={2126--2154},
   issn={0360-5302},
   review={\MR{3169773}},
   doi={10.1080/03605302.2013.828069},
}

\bib{2019arXiv190702495A}{article}{
       author = {Alibaud, Natha{\"e}l},
author={del Teso, F{\'e}lix},
author={Endal, J{\o}rgen},
author={Jakobsen, Espen R.},
        title = {The Liouville theorem and linear operators satisfying the maximum principle},
      journal = {arXiv e-prints},
date = {2019},
          eid = {arXiv:1907.02495},
        pages = {arXiv:1907.02495},
archivePrefix = {arXiv},
       eprint = {1907.02495},
 primaryClass = {math.AP},
       adsurl = {https://ui.adsabs.harvard.edu/abs/2019arXiv190702495A},
      adsnote = {Provided by the SAO/NASA Astrophysics Data System}
}

\bib{MR2601079}{article}{
   author={Apreutesei, Narcisa},
   author={Bessonov, Nikolai},
   author={Volpert, Vitaly},
   author={Vougalter, Vitali},
   title={Spatial structures and generalized travelling waves for an
   integro-differential equation},
   journal={Discrete Contin. Dyn. Syst. Ser. B},
   volume={13},
   date={2010},
   number={3},
   pages={537--557},
   issn={1531-3492},
   review={\MR{2601079}},
   doi={10.3934/dcdsb.2010.13.537},
}

\bib{MR2911421}{article}{
   author={Barles, Guy},
   author={Chasseigne, Emmanuel},
   author={Ciomaga, Adina},
   author={Imbert, Cyril},
   title={Lipschitz regularity of solutions for mixed integro-differential
   equations},
   journal={J. Differential Equations},
   volume={252},
   date={2012},
   number={11},
   pages={6012--6060},
   issn={0022-0396},
   review={\MR{2911421}},
   doi={10.1016/j.jde.2012.02.013},
}

\bib{MR3194684}{article}{
   author={Barles, Guy},
   author={Chasseigne, Emmanuel},
   author={Ciomaga, Adina},
   author={Imbert, Cyril},
   title={Large time behavior of periodic viscosity solutions for uniformly
   parabolic integro-differential equations},
   journal={Calc. Var. Partial Differential Equations},
   volume={50},
   date={2014},
   number={1-2},
   pages={283--304},
   issn={0944-2669},
   review={\MR{3194684}},
   doi={10.1007/s00526-013-0636-2},
}

\bib{MR2422079}{article}{
   author={Barles, Guy},
   author={Imbert, Cyril},
   title={Second-order elliptic integro-differential equations: viscosity
   solutions' theory revisited},
   journal={Ann. Inst. H. Poincar\'{e} Anal. Non Lin\'{e}aire},
   volume={25},
   date={2008},
   number={3},
   pages={567--585},
   issn={0294-1449},
   review={\MR{2422079}},
   doi={10.1016/j.anihpc.2007.02.007},
}

\bib{MR2095633}{article}{
   author={Bass, Richard F.},
   author={Kassmann, Moritz},
   title={Harnack inequalities for non-local operators of variable order},
   journal={Trans. Amer. Math. Soc.},
   volume={357},
   date={2005},
   number={2},
   pages={837--850},
   issn={0002-9947},
   review={\MR{2095633}},
   doi={10.1090/S0002-9947-04-03549-4},
}

\bib{MR2180302}{article}{
   author={Bass, Richard F.},
   author={Kassmann, Moritz},
   title={H\"{o}lder continuity of harmonic functions with respect to operators
   of variable order},
   journal={Comm. Partial Differential Equations},
   volume={30},
   date={2005},
   number={7-9},
   pages={1249--1259},
   issn={0360-5302},
   review={\MR{2180302}},
   doi={10.1080/03605300500257677},
}

\bib{MR3498523}{article}{
   author={Berestycki, Henri},
   author={Coville, J\'{e}r\^{o}me},
   author={Vo, Hoang-Hung},
   title={Persistence criteria for populations with non-local dispersion},
   journal={J. Math. Biol.},
   volume={72},
   date={2016},
   number={7},
   pages={1693--1745},
   issn={0303-6812},
   review={\MR{3498523}},
   doi={10.1007/s00285-015-0911-2},
}

\bib{biagvecc}{article}{
       author = {Biagi, Stefano},
author={Dipierro, Serena},
author={Valdinoci, Enrico},
author={Vecchi, Eugenio},
        title = {Mixed local and nonlocal elliptic operators:
regularity and maximum principles},
      journal = {arXiv e-prints},
date = {2020},
          eid = {arXiv:2005.06907},
        pages = {arXiv:2005.06907},
archivePrefix = {arXiv},
       eprint = {2005.06907},
       adsurl = {https://ui.adsabs.harvard.edu/abs/2020arXiv200506907B},
      adsnote = {Provided by the SAO/NASA Astrophysics Data System}
}

\bib{MR2653895}{article}{
   author={Biswas, Imran H.},
   author={Jakobsen, Espen R.},
   author={Karlsen, Kenneth H.},
   title={Viscosity solutions for a system of integro-PDEs and connections
   to optimal switching and control of jump-diffusion processes},
   journal={Appl. Math. Optim.},
   volume={62},
   date={2010},
   number={1},
   pages={47--80},
   issn={0095-4616},
   review={\MR{2653895}},
   doi={10.1007/s00245-009-9095-8},
}

\bib{PhysRevE.87.063106}{article}{
  author = {Blazevski, Daniel},
author={del-Castillo-Negrete, Diego},
  title = {Local and nonlocal anisotropic transport
in reversed shear magnetic fields: Shearless Cantori and nondiffusive transport},
  journal = {Phys. Rev. E},
  volume = {87},
  issue = {6},
  pages = {063106},
  numpages = {15},
  year = {2013},
  doi = {10.1103/PhysRevE.87.063106},
  url = {https://link.aps.org/doi/10.1103/PhysRevE.87.063106}
}

\bib{MR3639140}{article}{
   author={Bonnefon, Olivier},
   author={Coville, J\'{e}r\^{o}me},
   author={Legendre, Guillaume},
   title={Concentration phenomenon in some non-local equation},
   journal={Discrete Contin. Dyn. Syst. Ser. B},
   volume={22},
   date={2017},
   number={3},
   pages={763--781},
   issn={1531-3492},
   review={\MR{3639140}},
   doi={10.3934/dcdsb.2017037},
}

\bib{MR697382}{book}{
   author={Brezis, Ha\"{\i}m},
   title={Analyse fonctionnelle},
   language={French},
   series={Collection Math\'{e}matiques Appliqu\'{e}es pour la Ma\^{\i}trise. [Collection
   of Applied Mathematics for the Master's Degree]},
   note={Th\'{e}orie et applications. [Theory and applications]},
   publisher={Masson, Paris},
   date={1983},
   pages={xiv+234},
   isbn={2-225-77198-7},
   review={\MR{697382}},
}

\bib{MR576277}{article}{
   author={Brown, K. J.},
   author={Lin, S. S.},
   title={On the existence of positive eigenfunctions for an eigenvalue
   problem with indefinite weight function},
   journal={J. Math. Anal. Appl.},
   volume={75},
   date={1980},
   number={1},
   pages={112--120},
   issn={0022-247X},
   review={\MR{576277}},
   doi={10.1016/0022-247X(80)90309-1},
}

\bib{MR3485125}{article}{
   author={Cabr\'{e}, Xavier},
   author={Serra, Joaquim},
   title={An extension problem for sums of fractional Laplacians and 1-D
   symmetry of phase transitions},
   journal={Nonlinear Anal.},
   volume={137},
   date={2016},
   pages={246--265},
   issn={0362-546X},
   review={\MR{3485125}},
   doi={10.1016/j.na.2015.12.014},
}

\bib{CABRE}{article}{
   author={Cabr\'{e}, Xavier},
   author={Dipierro, Serena},
author={Valdinoci, Enrico},
title={The Bernstein technique for integro-differential
equations},
      journal = {preprint},
}

\bib{3579567}{article}{
   author={Caffarelli, Luis},
   author={Dipierro, Serena},
   author={Valdinoci, Enrico},
   title={A logistic equation with nonlocal interactions},
   journal={Kinet. Relat. Models},
   volume={10},
   date={2017},
   number={1},
   pages={141--170},
   issn={1937-5093},
   review={\MR{3579567}},
   doi={10.3934/krm.2017006},
}

\bib{MR3051400}{article}{
   author={Caffarelli, Luis},
   author={Valdinoci, Enrico},
   title={A priori bounds for solutions of a nonlocal evolution PDE},
   conference={
      title={Analysis and numerics of partial differential equations},
   },
   book={
      series={Springer INdAM Ser.},
      volume={4},
      publisher={Springer, Milan},
   },
   date={2013},
   pages={141--163},
   review={\MR{3051400}},
   doi={10.1007/978-88-470-2592-9\_10},
}

\bib{MR2332679}{article}{
   author={Cantrell, Robert Stephen},
   author={Cosner, Chris},
   author={Lou, Yuan},
   title={Advection-mediated coexistence of competing species},
   journal={Proc. Roy. Soc. Edinburgh Sect. A},
   volume={137},
   date={2007},
   number={3},
   pages={497--518},
   issn={0308-2105},
   review={\MR{2332679}},
   doi={10.1017/S0308210506000047},
}

\bib{MR3026598}{article}{
   author={Cantrell, Robert Stephen},
   author={Cosner, Chris},
   author={Lou, Yuan},
   author={Ryan, Daniel},
   title={Evolutionary stability of ideal free dispersal strategies: a
   nonlocal dispersal model},
   journal={Can. Appl. Math. Q.},
   volume={20},
   date={2012},
   number={1},
   pages={15--38},
   issn={1073-1849},
   review={\MR{3026598}},
}

\bib{MR2411225}{article}{
   author={Chen, Xinfu},
   author={Hambrock, Richard},
   author={Lou, Yuan},
   title={Evolution of conditional dispersal: a reaction-diffusion-advection
   model},
   journal={J. Math. Biol.},
   volume={57},
   date={2008},
   number={3},
   pages={361--386},
   issn={0303-6812},
   review={\MR{2411225}},
   doi={10.1007/s00285-008-0166-2},
}

\bib{MR2928344}{article}{
   author={Chen, Zhen-Qing},
   author={Kim, Panki},
   author={Song, Renming},
   author={Vondra\v{c}ek, Zoran},
   title={Sharp Green function estimates for $\Delta+\Delta^{\alpha/2}$ in
   $C^{1,1}$ open sets and their applications},
   journal={Illinois J. Math.},
   volume={54},
   date={2010},
   number={3},
   pages={981--1024 (2012)},
   issn={0019-2082},
   review={\MR{2928344}},
}

\bib{MR2912450}{article}{
   author={Chen, Zhen-Qing},
   author={Kim, Panki},
   author={Song, Renming},
   author={Vondra\v{c}ek, Zoran},
   title={Boundary Harnack principle for $\Delta+\Delta^{\alpha/2}$},
   journal={Trans. Amer. Math. Soc.},
   volume={364},
   date={2012},
   number={8},
   pages={4169--4205},
   issn={0002-9947},
   review={\MR{2912450}},
   doi={10.1090/S0002-9947-2012-05542-5},
}

\bib{MR2963799}{article}{
   author={Ciomaga, Adina},
   title={On the strong maximum principle for second-order nonlinear
   parabolic integro-differential equations},
   journal={Adv. Differential Equations},
   volume={17},
   date={2012},
   number={7-8},
   pages={635--671},
   issn={1079-9389},
   review={\MR{2963799}},
}

\bib{MR2897881}{article}{
   author={Cosner, Chris},
   author={D\'{a}vila, Juan},
   author={Mart\'{\i}nez, Salome},
   title={Evolutionary stability of ideal free nonlocal dispersal},
   journal={J. Biol. Dyn.},
   volume={6},
   date={2012},
   number={2},
   pages={395--405},
   issn={1751-3758},
   review={\MR{2897881}},
   doi={10.1080/17513758.2011.588341},
}

\bib{MR3285831}{article}{
   author={Coville, J\'{e}r\^{o}me},
   title={Nonlocal refuge model with a partial control},
   journal={Discrete Contin. Dyn. Syst.},
   volume={35},
   date={2015},
   number={4},
   pages={1421--1446},
   issn={1078-0947},
   review={\MR{3285831}},
   doi={10.3934/dcds.2015.35.1421},
}

\bib{MR3035974}{article}{
   author={Coville, J\'{e}r\^{o}me},
   author={D\'{a}vila, Juan},
   author={Mart\'{\i}nez, Salom\'{e}},
   title={Pulsating fronts for nonlocal dispersion and KPP nonlinearity},
   journal={Ann. Inst. H. Poincar\'{e} Anal. Non Lin\'{e}aire},
   volume={30},
   date={2013},
   number={2},
   pages={179--223},
   issn={0294-1449},
   review={\MR{3035974}},
   doi={10.1016/j.anihpc.2012.07.005},
}

\bib{defi}{article}{
   author={de Figueiredo, Djairo Guedes},
   title={Positive solutions of semilinear elliptic problems},
   conference={
      title={Differential equations},
      address={S\~ao Paulo},
      date={1981},
   },
   book={
      series={Lecture Notes in Math.},
      volume={957},
      publisher={Springer, Berlin-New York},
   },
   date={1982},
   pages={34--87},
   review={\MR{679140}},
}

\bib{MR2542727}{article}{
   author={de la Llave, Rafael},
   author={Valdinoci, Enrico},
   title={A generalization of Aubry-Mather theory to partial differential
   equations and pseudo-differential equations},
   journal={Ann. Inst. H. Poincar\'{e} Anal. Non Lin\'{e}aire},
   volume={26},
   date={2009},
   number={4},
   pages={1309--1344},
   issn={0294-1449},
   review={\MR{2542727}},
   doi={10.1016/j.anihpc.2008.11.002},
}

\bib{2018arXiv181107667D}{article}{
       author = {Dell'Oro, Filippo},
author={Pata, Vittorino},
        title = {Second order linear evolution equations with general dissipation},
      journal = {arXiv e-prints},
         date = {2018},
          eid = {arXiv:1811.07667},
        pages = {arXiv:1811.07667},
archivePrefix = {arXiv},
       eprint = {1811.07667},
 primaryClass = {math.AP},
       adsurl = {https://ui.adsabs.harvard.edu/abs/2018arXiv181107667D},
      adsnote = {Provided by the SAO/NASA Astrophysics Data System}
}

\bib{2017arXiv170605306D}{article}{
       author = {del Teso, F{\'e}lix},
author = {Endal, J{\o}rgen},
author = {Jakobsen, Espen R.},
        title = {On distributional solutions of local and nonlocal problems of porous medium type},
      journal = {arXiv e-prints},
date = {2017},
          eid = {arXiv:1706.05306},
        pages = {arXiv:1706.05306},
archivePrefix = {arXiv},
       eprint = {1706.05306},
 primaryClass = {math.AP},
       adsurl = {https://ui.adsabs.harvard.edu/abs/2017arXiv170605306D},
      adsnote = {Provided by the SAO/NASA Astrophysics Data System}
}

\bib{MR3237774}{article}{
   author={Di Castro, Agnese},
   author={Kuusi, Tuomo},
   author={Palatucci, Giampiero},
   title={Nonlocal Harnack inequalities},
   journal={J. Funct. Anal.},
   volume={267},
   date={2014},
   number={6},
   pages={1807--1836},
   issn={0022-1236},
   review={\MR{3237774}},
   doi={10.1016/j.jfa.2014.05.023},
}

\bib{MR3542614}{article}{
   author={Di Castro, Agnese},
   author={Kuusi, Tuomo},
   author={Palatucci, Giampiero},
   title={Local behavior of fractional $p$-minimizers},
   journal={Ann. Inst. H. Poincar\'{e} Anal. Non Lin\'{e}aire},
   volume={33},
   date={2016},
   number={5},
   pages={1279--1299},
   issn={0294-1449},
   review={\MR{3542614}},
   doi={10.1016/j.anihpc.2015.04.003},
}

\bib{MR2944369}{article}{
   author={Di Nezza, Eleonora},
   author={Palatucci, Giampiero},
   author={Valdinoci, Enrico},
   title={Hitchhiker's guide to the fractional Sobolev spaces},
   journal={Bull. Sci. Math.},
   volume={136},
   date={2012},
   number={5},
   pages={521--573},
   issn={0007-4497},
   review={\MR{2944369}},
   doi={10.1016/j.bulsci.2011.12.004},
}

\bib{VERO}{article}{
   author={Dipierro, Serena},
   author={Proietti Lippi, Edoardo},
   author={Valdinoci, Enrico},
title={(Non)local logistic equations with Neumann conditions},
      journal = {preprint},
}

\bib{MR3651008}{article}{
   author={Dipierro, Serena},
   author={Ros-Oton, Xavier},
   author={Valdinoci, Enrico},
   title={Nonlocal problems with Neumann boundary conditions},
   journal={Rev. Mat. Iberoam.},
   volume={33},
   date={2017},
   number={2},
   pages={377--416},
   issn={0213-2230},
   review={\MR{3651008}},
   doi={10.4171/RMI/942},
}

\bib{MR3950697}{article}{
   author={Dipierro, Serena},
   author={Valdinoci, Enrico},
   author={Vespri, Vincenzo},
   title={Decay estimates for evolutionary equations with fractional
   time-diffusion},
   journal={J. Evol. Equ.},
   volume={19},
   date={2019},
   number={2},
   pages={435--462},
   issn={1424-3199},
   review={\MR{3950697}},
   doi={10.1007/s00028-019-00482-z},
}

\bib{MR1636644}{article}{
   author={Dockery, Jack},
   author={Hutson, Vivian},
   author={Mischaikow, Konstantin},
   author={Pernarowski, Mark},
   title={The evolution of slow dispersal rates: a reaction diffusion model},
   journal={J. Math. Biol.},
   volume={37},
   date={1998},
   number={1},
   pages={61--83},
   issn={0303-6812},
   review={\MR{1636644}},
   doi={10.1007/s002850050120},
}

\bib{MR2597943}{book}{
   author={Evans, Lawrence C.},
   title={Partial differential equations},
   series={Graduate Studies in Mathematics},
   volume={19},
   edition={2},
   publisher={American Mathematical Society, Providence, RI},
   date={2010},
   pages={xxii+749},
   isbn={978-0-8218-4974-3},
   review={\MR{2597943}},
   doi={10.1090/gsm/019},
}

\bib{MR1911531}{book}{
   author={Garroni, Maria Giovanna},
   author={Menaldi, Jose Luis},
   title={Second order elliptic integro-differential problems},
   series={Chapman \& Hall/CRC Research Notes in Mathematics},
   volume={430},
   publisher={Chapman \& Hall/CRC, Boca Raton, FL},
   date={2002},
   pages={xvi+221},
   isbn={1-58488-200-X},
   review={\MR{1911531}},
   doi={10.1201/9781420035797},
}

\bib{MR1669352}{book}{
   author={Han, Qing},
   author={Lin, Fanghua},
   title={Elliptic partial differential equations},
   series={Courant Lecture Notes in Mathematics},
   volume={1},
   publisher={New York University, Courant Institute of Mathematical
   Sciences, New York; American Mathematical Society, Providence, RI},
   date={1997},
   pages={x+144},
   isbn={0-9658703-0-8},
   isbn={0-8218-2691-3},
   review={\MR{1669352}},
}

\bib{MR3593528}{article}{
   author={Iannizzotto, Antonio},
   author={Mosconi, Sunra},
   author={Squassina, Marco},
   title={Global H\"{o}lder regularity for the fractional $p$-Laplacian},
   journal={Rev. Mat. Iberoam.},
   volume={32},
   date={2016},
   number={4},
   pages={1353--1392},
   issn={0213-2230},
   review={\MR{3593528}},
   doi={10.4171/RMI/921},
}

\bib{MR2129093}{article}{
   author={Jakobsen, Espen R.},
   author={Karlsen, Kenneth H.},
   title={Continuous dependence estimates for viscosity solutions of
   integro-PDEs},
   journal={J. Differential Equations},
   volume={212},
   date={2005},
   number={2},
   pages={278--318},
   issn={0022-0396},
   review={\MR{2129093}},
   doi={10.1016/j.jde.2004.06.021},
}

\bib{MR2243708}{article}{
   author={Jakobsen, Espen R.},
   author={Karlsen, Kenneth H.},
   title={A ``maximum principle for semicontinuous functions'' applicable to
   integro-partial differential equations},
   journal={NoDEA Nonlinear Differential Equations Appl.},
   volume={13},
   date={2006},
   number={2},
   pages={137--165},
   issn={1021-9722},
   review={\MR{2243708}},
   doi={10.1007/s00030-005-0031-6},
}

\bib{MR2924452}{article}{
   author={Kao, Chiu-Yen},
   author={Lou, Yuan},
   author={Shen, Wenxian},
   title={Evolution of mixed dispersal in periodic environments},
   journal={Discrete Contin. Dyn. Syst. Ser. B},
   volume={17},
   date={2012},
   number={6},
   pages={2047--2072},
   issn={1531-3492},
   review={\MR{2924452}},
   doi={10.3934/dcdsb.2012.17.2047},
}

\bib{MR3590678}{article}{
   author={Massaccesi, Annalisa},
   author={Valdinoci, Enrico},
   title={Is a nonlocal diffusion strategy convenient for biological
   populations in competition?},
   journal={J. Math. Biol.},
   volume={74},
   date={2017},
   number={1-2},
   pages={113--147},
   issn={0303-6812},
   review={\MR{3590678}},
   doi={10.1007/s00285-016-1019-z},
}

\bib{MR3082317}{article}{
   author={Montefusco, Eugenio},
   author={Pellacci, Benedetta},
   author={Verzini, Gianmaria},
   title={Fractional diffusion with Neumann boundary conditions: the
   logistic equation},
   journal={Discrete Contin. Dyn. Syst. Ser. B},
   volume={18},
   date={2013},
   number={8},
   pages={2175--2202},
   issn={1531-3492},
   review={\MR{3082317}},
   doi={10.3934/dcdsb.2013.18.2175},
}

\bib{MR3771424}{article}{
   author={Pellacci, Benedetta},
   author={Verzini, Gianmaria},
   title={Best dispersal strategies in spatially heterogeneous environments:
   optimization of the principal eigenvalue for indefinite fractional
   Neumann problems},
   journal={J. Math. Biol.},
   volume={76},
   date={2018},
   number={6},
   pages={1357--1386},
   issn={0303-6812},
   review={\MR{3771424}},
   doi={10.1007/s00285-017-1180-z},
}

\bib{MR3060890}{article}{
   author={Servadei, Raffaella},
   author={Valdinoci, Enrico},
   title={A Brezis-Nirenberg result for non-local critical equations in low
   dimension},
   journal={Commun. Pure Appl. Anal.},
   volume={12},
   date={2013},
   number={6},
   pages={2445--2464},
   issn={1534-0392},
   review={\MR{3060890}},
   doi={10.3934/cpaa.2013.12.2445},
}

\bib{MR3161511}{article}{
   author={Servadei, Raffaella},
   author={Valdinoci, Enrico},
   title={Weak and viscosity solutions of the fractional Laplace equation},
   journal={Publ. Mat.},
   volume={58},
   date={2014},
   number={1},
   pages={133--154},
   issn={0214-1493},
   review={\MR{3161511}},
}

\bib{MR3590646}{article}{
   author={Sprekels, J\"{u}rgen},
   author={Valdinoci, Enrico},
   title={A new type of identification problems: optimizing the fractional
   order in a nonlocal evolution equation},
   journal={SIAM J. Control Optim.},
   volume={55},
   date={2017},
   number={1},
   pages={70--93},
   issn={0363-0129},
   review={\MR{3590646}},
   doi={10.1137/16M105575X},
}

\bib{NATU}{article}{
   author={Viswanathan, G. M.},
   author={Afanasyev, V.},
   author={Buldyrev, S. V.},
   author={Murphy, E. J.},
   author={Prince, P. A.},
   author={Stanley, H. E.},
   title={L\'evy flight search patterns of wandering albatrosses},
   journal={Nature},
   volume={381},
   date={1996},
   pages={413--415},
   issn={1476-4687},
   doi={10.1038/381413a0},
}

\end{biblist}\end{bibdiv}
\end{document}